\documentclass{amsart}
\usepackage{amssymb}

\usepackage[usenames]{color}

\input xy
\xyoption{all}

\newcommand{\Tych}{\mathbf{Tych}}

\newcommand{\IR}{\mathbb R}
\newcommand{\IN}{\mathbb N}
\newcommand{\IQ}{\mathbb Q}

\newcommand{\X}{\mathcal{X}}
\newcommand{\Y}{\mathcal{Y}}
\newcommand{\Z}{\mathcal{Z}}
\newcommand{\id}{\mathrm{id}}

\newcommand{\M}{\mathfrak M}

\newcommand{\w}{\omega}
\newcommand{\supp}{\mathrm{supp}}

\newcommand{\CC}{C_k}
\newcommand{\cp}{\mathbf{C}_p}
\newcommand{\XN}{X^{<\omega}}

\newcommand{\DN}{D^{<\omega}}
\newcommand{\MN}{\mathfrak{M}^{<\omega}}

\newcommand{\scp}{\mathbf{SC}_p}

\newcommand{\sep}{\mathsf{sp}}
\newcommand{\Top}{\mathfrak{T}}
\newcommand{\Lin}{\mathfrak{L}}

\newcommand{\NN}{\mathbb{N}}

\newcommand{\Pp}{\mathfrak{P}}
\newcommand{\Nn}{\mathcal{N}}

\newtheorem{theorem}{Theorem}[section]
\newtheorem{lemma}[theorem]{Lemma}
\newtheorem{corollary}[theorem]{Corollary}

\newtheorem{fact}[theorem]{Fact}
\newtheorem{problem}[theorem]{Problem}

\theoremstyle{definition}
\newtheorem{definition}[theorem]{Definition}

\newtheorem{example}[theorem]{Example}

\title[The $C_p$-stable closure of the class of separable metrizable spaces]{The $C_p$-stable  closure of the class\\ of separable metrizable spaces}
\author{T. Banakh,  S. Gabriyelyan}

\address{T. Banakh: Ivan Franko National University of Lviv (Ukraine) and Jan Kochanowski University in Kielce (Poland)}
\email{t.o.banakh@gmail.com}
\address{S. Gabriyelyan: Department of Mathematics, Ben-Gurion University of the Negev, Beer-Sheva, P.O. 653, Israel}
\email{saak@math.bgu.ac.il}
\begin{document}

\begin{abstract}
Denote by $\cp[\M_0]$ the $C_p$-stable closure of the class $\M_0$ of all separable metrizable spaces, i.e., $\cp[\M_0]$ is  the smallest class of topological spaces that contains $\M_0$ and is closed under taking subspaces, homeomorphic images, countable topological sums, countable Tychonoff products, and function spaces $C_p(X,Y)$. Using a recent deep result of Chernikov and Shelah (2014), we prove that $\cp[\M_0]$ coincides with the class of all Tychonoff spaces of cardinality strictly less than $\beth_{\w_1}$. Being motivated by the theory of Generalized Metric Spaces, we characterize also other natural $C_p$-type stable  closures  of the class $\M_0$.
\end{abstract}

\maketitle

\section{Introduction}

All topological spaces considered in this paper are assumed to be Tychonoff. Many natural and important classes $\X$ of topological spaces usually are preserved by basic topological operations, which motivated us in \cite{BG} to  introduce the following stability  property of $\X$. Following  \cite{BG}, we say that a class $\X$ of topological spaces is {\em stable} if $\X$ is closed under taking subspaces, homeomorphic images, countable topological sums, and countable Tychonoff products. For two classes $\X$ and $\tilde\X$ of topological spaces, we shall say that $\tilde\X$ is an {\em extension} of $\X$ if $\X\subset\tilde\X$. Among all stable extensions of $\X$ there is the smallest one  denoted by $[\X]$ and called the {\em stable closure} of the class $\X$.

It is clear that the class $\M_0$ of all separable metrizable spaces is stable. One of the most natural generalizations of the class $\M_0$ is the class of all cosmic spaces. Recall (see \cite{gruenhage}) that a topological space $X$ is called {\em cosmic} (a {\em $\sigma$-space}) if it is regular and possesses a countable (respectively, a $\sigma$-locally finite) network. A family $\mathcal{N}$ of subsets of $X$ is a {\em network} for $X$ if for any $x\in X$ and any open neighborhood $U$ of $x$ there is a set $N\in\mathcal{N}$ such that $x\in N\subset U$. The class $\mathfrak{C}$ of all cosmic spaces is stable  \cite{Mich} and so are many other classes of generalized metric spaces, in particular, the classes of $\sigma$-spaces, $\aleph_0$-spaces, $\aleph$-spaces, $\Pp_0$-spaces and $\Pp$-spaces (see, \cite{Mich, gruenhage, Banakh,GK-GMS1}). Clearly, the class $\Top$ of all topological spaces is stable, but the class $\Lin$ of all Lindel\"{o}f spaces is not stable. Note that any cosmic space is Lindel\"{o}f.

For two topological spaces $X$ and $Y$, we denote by  $\CC(X,Y)$ and $C_p(X,Y)$ the space $C(X,Y)$ of all continuous functions  from $X$ into $Y$ endowed with the compact-open topology and  the topology of pointwise convergence, respectively. 
Having in mind the stability of the classes of $\aleph_0$-spaces and $\Pp_0$-spaces under taking function spaces $\CC(X,Y)$ (see \cite{Mich,Banakh}), we defined  in \cite{BG}  a class $\Z$ of topological spaces to be {\em $\CC$-stable} if $\Z$ is stable and for any Lindel\"{o}f space $X\in\Z$ and any space $Y\in\Z$ the function space $\CC(X,Y)$ belongs to the class $\Z$. In  \cite{BG} we proved  that the smallest $\CC$-stable extension of the class $\M_0$ coincides with the class $\mathbf C_k(\M_0,\M_0)$ of all topological spaces which can be embedded into the function spaces $\CC(X,Y)$ for suitable separable metric spaces $X$ and $Y$.   These results  motivate the following definition.
\begin{definition}
Let $\X,\Y,\Z$ be classes of topological spaces. The class $\Z$ is called {\em $C_p^{\X\kern-1pt,\Y}$-stable} if $\Z$ is stable and for any spaces $X\in\Z\cap\X$ and $Y\in\Z\cap\Y$ the function space $C_p(X,Y)$ belongs to the class $\Z$.  The smallest $C_p^{\X\kern-1pt,\Y}$-stable extension of the class $\Z$ is denoted by $\cp^{\X\kern-1pt,\Y}[\Z]$ and called the {\em $C_p^{\X\kern-1pt,\Y}$-stable closure} of the class $\Z$.
\end{definition}
In this paper we shall describe the $C_p^{\X\kern-1pt,\Y}$-stable closures of the class $\M_0$ for the cases $\X$ and $\Y$ equal to $\M_0$ or $\Top$, the class of all topological spaces.

In case $\X=\Y=\Top$ the class  $\cp^{\Top\kern-1pt,\Top}[\Z]$ will be denoted by $\cp[\Z]$ and called the {\em $C_p$-stable closure} of $\Z$.

Our first theorem (which will be proved in Section~\ref{s2}) essentially follows from Michael's results \cite{Mich}.

\begin{theorem}\label{t:Cp-M0-M0}
The classes $\cp^{\Top\kern-1pt,\M_0}[\M_0]$ and $\cp^{\M_0\kern-1pt,\M_0}[\M_0]$ coincide with the class of all cosmic spaces.
\end{theorem}

The class $\cp^{\M_0\kern-1pt,\Top}[\M_0]$ is strictly larger than $\cp^{\M_0\kern-1pt,\M_0}[\M_0]=\cp^{\Top\kern-1pt,\M_0}[\M_0]$ and admits the following description. For a topological space $X$, let $X^{<\w}=\bigoplus_{n\in\w}X^n$ be the topological sum of all its finite powers. By $SC_p(X^{<\w})$ we denote the space of all real-valued separately continuous functions $f:X^{<\w}\to\IR$, endowed with the topology of pointwise convergence. For a class $\X$ of topological spaces, we denote  by $\scp(\X^{<\w})$ the class of topological spaces which embed into  function spaces $SC_p(X^{<\w})$ with $X\in\X$.

\begin{theorem}\label{t:Cp-M0-Top}
The class $\cp^{\M_0\kern-1pt,\Top}[\M_0]$ coincides with the class $\scp(\M_0^{<\w})$.
\end{theorem}

Theorem~\ref{t:Cp-M0-Top} motivates studying the class $\scp(\M_0^{<\w})$ in more details.
It turns out that this class contains all cosmic spaces, all metrizable spaces of weight $\le\mathfrak c$ and also all generalized ordered spaces whose order topology is second countable. Let us recall the necessary definitions related to generalized ordered spaces.

By a {\em linearly ordered space} we understand a linearly ordered set $(X,\leq)$ endowed with the {\em order topology} generated by the subbase consisting of open half-intervals $(\leftarrow,a)=\{x\in X:x<a\}$ and $(a,\rightarrow)=\{x\in X:x>a\}$ where $a\in X$. By \cite[3.12.4(d)]{Eng}, the weight $w(X)$ of a linearly ordered space $X$ coincides with its network weight $nw(X)$. A subset $C\subset X$ is {\em order convex} if for any points $x\leq y$ in $C$ the order interval $[x,y]=\{z\in X:x\leq z\leq y\}$ is contained in $C$.

A {\em generalized ordered space} (briefly, a {\em $GO$-space}) is a topological space $X$ endowed with a linear order $\leq$ such that open order convex subsets of $X$ form a base of the topology of $X$. Standard examples of $GO$-spaces are the Sorgenfrey line and the Michael line (see \cite[1.2.2 and 5.1.32]{Eng}). Also the real line $\IR_d$ endowed with the discrete topology is a $GO$-space.

By the {\em order topology} on a $GO$-space $X$ we understand the topology on $X$ generated by the linear order. The {\em linear weight} $lw(X)$ of a $GO$-space $X$ is the weight of its order topology.
In particular, $lw(\IR_d)=\w$. It follows that $w(X)\geq nw(X)\geq lw(X)$ for each $GO$-space $X$. Note that $w(\IR_d)=\mathfrak{c}> lw(\IR_d)=\aleph_0$. 

\begin{theorem}\label{t:lower} The class $\scp(\M_0^{<\w})$ contains all cosmic spaces, all metrizable spaces of weight $\le\mathfrak c$, and all generalized ordered spaces of countable linear weight.
\end{theorem}

An upper bound on the class $\scp(\M_0^{<\w})$ is given by the class of spaces with countable $i$-weight. We recall that the {\em $i$-weight} $iw(X)$ of a Tychonoff space $X$ is the smallest cardinal $\kappa$ for which there is an injective continuous map $i:X\to Y$ into a Tychonoff space $Y$ of weight $w(Y)\leq\kappa$. It is clear that $iw(X)\leq w(X)$, and if $iw(X)\leq\aleph_0$ then $|X|\leq \mathfrak{c}$. The following theorem, proved in Section~\ref{s:up}, gives an upper bound on the class $\scp(\M_0^{<\w})$.

\begin{theorem} \label{t:Cp(M)-iw} Each space $X\in \cp^{\M_0\kern-1pt,\Top}[\M_0]=\scp(\M_0^{<\w})$ has weight $w(X)\le\mathfrak c$, cardinality $|X|\le\mathfrak c$ and $i$-weight $iw(X)\le\aleph_0$.
\end{theorem}

\begin{example}
By \cite{GKKM}, the Banach space $X=\ell_1(\mathfrak c)$ endowed with the weak topology is an $\aleph$-space of cardinality $|X|=\mathfrak c$, weight $w(X)=2^{\mathfrak c}$ and $i$-weight $iw(X)=\aleph_0$. By Theorem~\ref{t:Cp(M)-iw}, the space $X$ does not belong to the class $\cp^{\M_0\kern-1pt,\Top}[\M_0]=\scp(\M_0^{<\w})$. Let us recall \cite[\S11]{gruenhage} that a regular topological space $X$ is an {\em $\aleph$-space} if $X$ has a $\sigma$-locally finite $k$-network. A family $\Nn$ of subsets of $X$ is called a {\em $k$-network} in $X$  if for any open set $U\subset X$ and compact subset $K\subset U$ there is a finite subfamily  $\mathcal{F} \subset\mathcal{N}$ such that $K\subset \bigcup \mathcal{F}\subset U$.
\end{example}

Since  each compact space of countable $i$-weight is metrizable, Theorem~\ref{t:Cp(M)-iw} implies:
\begin{corollary}\label{c:Cp(M)-Com}
Each compact space in the class $\cp^{\M_0\kern-1pt,\Top}[\M_0]=\scp(\M_0^{<\w})$
is metrizable.
\end{corollary}

This corollary  shows that the class $\cp^{\M_0\kern-1pt,\Top}[\M_0]$ can be considered as a new class of generalized metric spaces.

Now we describe the $C_p$-stable closure $\cp[\M_0]:=\cp^{\Top\kern-1pt,\Top}[\M_0]$ of the class $\M_0$.  It turns out that the class $\cp[\M_0]$ is huge and coincides with the class of all  Tychonoff spaces of cardinality (or weight) strictly smaller than $\beth_{\w_1}$. Here $\beth_0=\aleph_0$ and $\beth_\alpha=\sup\{2^{\beth_\beta}:\beta<\alpha\}$ for any ordinal $\alpha>0$.

\begin{theorem}\label{t:Cp-Top-Top}
The $C_p$-stable closure $\cp[\M_0]$ of the class $\M_0$ of separable metrizable spaces coincides with the class of all Tychonoff spaces of cardinality strictly less than $\beth_{\w_1}$.
\end{theorem}
We prove this theorem in Section~\ref{s6} using a recent deep result of Chernikov and Shelah \cite{CS}.

\section{Proof of Theorem~\ref{t:Cp-M0-M0}}\label{s2}

We shall deduce Theorem~\ref{t:Cp-M0-M0} from the  following characterization of cosmic spaces due to  Michael \cite{Mich}.

\begin{fact}[\cite{Mich}] \label{f:Cosmic} For a Tychonoff space $X$ the following conditions are equivalent:
\begin{enumerate}
\item $X$ is cosmic;
\item $X$ is  a continuous image of a separable metrizable space;
\item the function space $C_p(X):=C_p(X;\IR)$ is cosmic;
\item for any metrizable separable space $Y$ the function space $C_p(X,Y)$ is cosmic.
\end{enumerate}
\end{fact}

For a Tychonoff space $X$, we denote by $\delta:X\to C_p(C_p(X))$ the canonical map assigning to each point $x\in X$ the Dirac measure $\delta_x$ concentrated at $x$. The {\em Dirac measure} $\delta_x:C_p(X)\to\IR$ assigns to each function $f\in C_p(X)$ its value $f(x)$ at the point $x$. The following important fact is well-known and can be found in  \cite[0.5.5]{Arhan}.

\begin{fact} \label{f:CpCp(X)}
For any Tychonoff space $X$, the canonical map $\delta:X\to C_p(C_p(X))$  is a topological embedding.
\end{fact}

Now we  present a proof of Theorem  \ref{t:Cp-M0-M0}.
\begin{proof}[Proof of Theorem  \ref{t:Cp-M0-M0}]
Let $X$ be a cosmic space. By Fact \ref{f:Cosmic}, the space $C_p(X)$ is the image of a separable metrizable space $M$ under a continuous map $\xi: M\to C_p(X)$. It follows that the dual map $\xi^\ast: C_p(C_p(X)) \to C_p(M)$, $\xi^\ast(f)=f\circ\xi$, is a topological embedding of $C_p(C_p(X))$ into $C_p(M)\in\cp^{\M_0\kern-1pt,\M_0}[\M_0]$. Now Fact \ref{f:CpCp(X)} implies that $X\in \cp^{\M_0\kern-1pt,\M_0}[\M_0]\subset  \cp^{\Top\kern-1pt,\M_0}[\M_0]$. Therefore, the class $\cp^{\M_0\kern-1pt,\M_0}[\M_0]$ contains all cosmic spaces.

To complete the proof of the theorem it is enough to show that the class $\mathfrak{C}$ of all cosmic spaces is $C_p^{\Top\kern-1pt,\M_0}$-stable. But this immediately follows from the stability of the class $\mathfrak{C}$ and Fact~\ref{f:Cosmic}.
\end{proof}

\section{Proof of Theorem~\ref{t:Cp-M0-Top}}

For topological spaces $X_1,\dots,X_n,Y,Z$, we denote by $SC_p( X_1\times\dots\times X_n,Y)$ the space of all separately continuous functions $f:X_1\times\dots\times X_n\to Y$ endowed with the topology of pointwise convergence. It is well-known that the function space $SC_p( X_1\times\dots\times X_n,Y)$ is  canonically homeomorphic  to  $C_p(X_1,SC_p(X_2\times\dots\times X_n,Y))$. In particular, $SC_p(X\times Y,Z)$ is canonically homeomorphic to $C_p(X,C_p(Y,Z))$. In the sequel the function space $SC_p(X_1\times\dots\times X_n,\IR)$ will be denoted by $SC_p(X_1\times\dots\times X_n)$. For $n=1$ the space $SC_p(X_1)$ coincides with $C_p(X_1)$.

We shall prove by induction that for every $n\in\IN$ and every space $X\in\M_0$ the function space $SC_p(X^n)$ belongs to the class $\cp^{\M_0,\Top}[\M_0]$. For $n=1$ the space $SC_p(X)=C_p(X)$ belongs to $\cp^{\M_0,\Top}[\M_0]$ by definition. Assume that for some number $n\in\IN$ we have proved that the space $SC_p(X^n)$ belongs to
$\cp^{\M_0,\Top}[\M_0]$. Taking into account that $SC_p(X^{n+1})$ is (canonically) homeomorphic to $C_p(X,SC_p(X^n))$, we conclude that  $SC_p(X^{n+1})\in \cp^{\M_0,\Top}[\M_0]$. Taking into account that the function space $SC_p(X^{<\w})$ is homeomorphic to $\prod_{n\in\IN}SC_p(X^n)$, we conclude that $SC_p(X^{<\w})\in \cp^{\M_0,\Top}[\M_0]$ and hence $SC_p(\M_0^{<\w})\subset \cp^{\M_0,\Top}[\M_0]$.

 To prove the reverse inclusion it is enough to check that the class $\scp\left(\MN_0\right)$ is closed under taking countable topological sums, countable Tychonoff products and taking function spaces with metrizable separable domain.
\smallskip

To see that the class $\scp\left(\MN_0\right)$ is closed under taking countable topological sums, it is enough to prove that for any non-empty spaces $X_n\in\M_0$, $n\in\w$, the topological sum $\bigoplus_{n\in\w}SC_p(\XN_n)$ embeds into the function space $SC_p(\XN)$ for some space $X\in\M_0$. Consider the topological sum $X=\bigoplus_{n\in\w}X_n$  and the topological embedding
\[
e:\textstyle{\bigoplus\limits_{n\in\w}}SC_p(\XN_n)\hookrightarrow SC_p(\XN)
\]
assigning to each function $f\in SC_p(\XN_n)$, $n\in\w$, the function $\tilde f\in SC_p(\XN)$ defined by
\[
\tilde f(x)=\begin{cases}f(x), &\mbox{if $x\in \XN_n$};\\
0 , & \mbox{if $x\in \XN\setminus\XN_n$.}
\end{cases}
\]
As $X\in\M_0$, we conclude that $ SC_p(\XN)$ and hence $\bigoplus_{n\in\w}SC_p(\XN_n)$ belong to the class $\scp\left(\MN_0\right)$.

Next we prove that for any metrizable space $X$ and any space $Y\in \scp\left(\MN_0\right)$ we get $C_p(X,Y)\in \scp\left(\MN_0\right)$.  The space $Y\in \scp\left(\MN_0\right)$ can be identified with a subspace of the function space $SC_p(Z^{<\w})$ for some separable metrizable space $Z$. Then
 the space $C_p(X,Y)$ can be identified with the subspace of $C_p(X,SC_p(Z^{<\w}))$, which is canonically homeomorphic to
\[
C_p\Big(X,\prod_{n\in\IN} SC_p (Z^n)\Big)= \prod_{n\in\NN} C_p(X,SC_p(Z^n))=\prod_{n\in\IN}SC_p(X\times Z^n).
\]
As $X\times Z^n$ is canonically homeomorphic to a retract $X\times\{x_0\}^{n-1}\times Z^n$ of $X^n\times Z^n =(X\times Z)^n$, we obtain
\[
C_p(X,SC_p(Z^{<\w}))=\prod_{n\in\NN} SC_p (X\times Z^n) \hookrightarrow \prod_{n\in\NN} SC_p ((X\times Z)^n)= SC_p \big((X\times Z)^{<\w}\big).
\]
Since  $X\times Z\in\M_0$, we finally get $C_p(X,Y) \hookrightarrow SC_p \left((X\times Z)^{<\w}\right)\in \scp\left(\MN_0\right)$.

Finally we show that the class $\scp\left(\MN_0\right)$ is closed under taking countable Tychonoff products. Fix any spaces $X_n\in \scp\left(\MN_0\right)$, $n\in\w$. As the class $\scp\left(\MN_0\right)$ is closed under taking countable topological sums, the topological sum $X=\bigoplus_{n\in\w}X_n$ belongs to the class $\scp\left(\MN_0\right)$. Since the class $\scp\left(\MN_0\right)$ is closed also under taking function spaces with separable metrizable domain, the function space $C_p(\w,X)$ belongs to the class $\scp\left(\MN_0\right)$. Taking into account that $\prod_{n\in\w}X_n\subset X^\w=C_p(\w,X)\in \scp\left(\MN_0\right)$, we conclude that $\prod_{n\in\w}X_n\in \scp\left(\MN_0\right)$.
\smallskip

\section{Proof of Theorem~\ref{t:lower}}

We divide the proof of Theorem~\ref{t:lower} into three lemmas.
The first of them follows from Theorem~\ref{t:Cp-M0-M0} and the obvious inclusion
$\cp^{\M_0,\M_0}[\M_0]\subset\cp^{\M_0,\Top}[\M_0]$.

\begin{lemma} \label{p:Cp-Cosmic}
The class $\cp^{\M_0\kern-1pt,\Top}[\M_0]$ contains all cosmic spaces.
\end{lemma}

The next lemma will play an important role also in the proof of Theorem~\ref{t:Cp-Top-Top}. In this lemma an ordered field is endowed with the topology generated by the linear order.

\begin{lemma}\label{lemma:SCp}
For every ordered field $F$, the space $SC_p(F\times F,F)\cong C_p(F,C_p(F,F))$ contains a discrete subspace $D$ of cardinality $|F|$.
\end{lemma}
\begin{proof}
Consider the discontinuous separately continuous function $\sep^F :F\times F\to F$ defined by
\[
\sep^F(x,y)=\begin{cases}
0, &\mbox{if $(x,y)=(0,0)$},\\
\frac{2xy}{x^2+y^2}, &\mbox{otherwise}.
\end{cases}
\]
For every $a,b\in F$, consider the shifted function $\sep^F_{a,b}:F\times F\to F$, $\sep^F_{a,b}(x,y):=\sep^F(x-a,y-b)$, and observe that the subspace $D=\{\sep^F_{a,a}:a\in F\}$ is discrete in the space $SC_p(F\times F,F)\cong C_p(F,C_p(F,F))$.
\end{proof}

\begin{lemma} \label{p:Cp-Metr}
 Each metrizable space $X$ of weight $\le\mathfrak c$ belongs to the class $\cp^{\M_0\kern-1pt,\Top}[\M_0]$.
\end{lemma}

\begin{proof}
Let $sp:=\sep^\IR :\IR\times\IR\to \IR$ be the classical discontinuous separately continuous function. By Lemma \ref{lemma:SCp},  the subspace $D=\{sp_{a,a}:a\in\IR\}$ is discrete in the space $SC_p(\IR\times\IR,\IR)$. Moreover, the subspace $H=\{t\cdot sp_{a,a}:a\in\IR,\;t\in[0,1]\}$ of $SC_p(\IR\times\IR,\IR)$ is homeomorphic to a hedgehog with $\mathfrak c$ spines. It follows from $SC_p(\IR\times\IR,\IR)\cong C_p(\IR,C_p(\IR))\in \cp^{\M_0\kern-1pt,\Top}[\M_0]$ that $H\in \cp^{\M_0\kern-1pt,\Top}[\M_0]$, and hence $H^\w\in \cp^{\M_0\kern-1pt,\Top}[\M_0]$. By \cite[4.4.9]{Eng}, each metrizable space of weight $\leq\mathfrak{c}$ embeds into $H^\w$. Consequently, the class $\cp^{\M_0\kern-1pt,\Top}[\M_0]$ contains all metrizable spaces of weight $\leq\mathfrak{c}$.
\end{proof}

\begin{lemma}
The class $\cp^{\M_0,\Top}[\M_0]$ contains all generalized ordered spaces of countable linear weight.
\end{lemma}

\begin{proof} Let $(X,\le,\tau)$ be a generalized ordered space with countable linear weight. Consider the following two subsets of $X$:
\[
X_\ell=\{x\in X:\mbox{$(\leftarrow,x]$ is open in $X$}\},\mbox{ and }
X_r=\{x\in X:\mbox{$[x,\rightarrow)$ is open in $X$}\}.
\]
Let $\tau_{\leq}$ be the topology on $X$ generated by the linear order.
By our assumption, the topology $\tau_{\leq}$ on $X$ is second countable, so, by \cite[6.3.2(c)]{Eng},  the linearly ordered space $(X,\tau_\leq)$ is order homeomorphic to a subspace of the real line $\IR$. Therefore we can assume that $X\subset \IR$ and the topology $\tau$ on $X$ is generated by the subbase
\[
\tau_\leq\cup\big\{(\leftarrow, x]:x\in X_\ell\big\}\cup\big\{[x,\rightarrow):x\in X_r\big\}.
\]
Denote by $\tilde X_\ell$ and $\tilde X_r$ the sets $X_\ell$ and $X_r$ respectively endowed with the  separable metrizable topology inherited from the real line $\IR$.

Besides the Euclidean topology $\tilde\tau_{\leq}$ generated by the linear order on the real line $\IR$, we shall consider the following three $GO$-topologies on $\IR$:
\begin{itemize}
\item the topology $\tilde\tau_\ell$ generated by the subbase $\tilde\tau_{\le}\cup\{(-\infty,x]:x\in X_\ell\}$;
\item the topology $\tilde\tau_r$ generated by the subbase $\tilde\tau_{\le}\cup\{[x,+\infty):x\in X_r\}$.
\item the topology $\tilde\tau$ generated by the subbase $\tilde\tau_\ell\cup\tilde\tau_r$.
\end{itemize}
It is easy to see that the identity map $\id:(X,\tau)\to(\IR,\tilde\tau)$ is a topological embedding. So, it suffices to show that $(\IR,\tilde\tau) \in SC_p(\M_0^{<\w})$.

For this purpose we consider the following two separately continuous functions $L,R:\IR\times\IR\to\IR$ defined by the formulas
\[
L(x,y):=\begin{cases}\frac{2xy}{x^2+y^2}&\mbox{if $x,y<0$},\\
0,&\mbox{otherwise,}
\end{cases}
\quad\mbox{
and
}\quad
R(x,y):=\begin{cases}\frac{2xy}{x^2+y^2}&\mbox{if $x,y>0$},\\
0,&\mbox{otherwise}.
\end{cases}
\]
For every $a\in\IR$, consider the shifted functions $L_a,R_a:\IR\times\IR\to\IR$ defined by
\[
L_a(x,y):=L(x-a,y-a) \; \mbox{ and } \; R_a(x,y):=R(x-a,y-a), \; \forall x,y\in\IR.
\]

Let $a\in\IR, \varepsilon>0$ and $(x,y)\in {\tilde X}_\ell\times {\tilde X}_\ell$. If $x<a$ and $y<a$, we can find $\delta >0$ such that
\begin{equation}\label{e:Cp-1}
|L_a(x,y)-L_b(x,y)|<\varepsilon, \quad \forall b\in(a-\delta,a+\delta)\in {\tilde\tau}_\leq .
\end{equation}
If $x>a$ or $y>a$, we can find $\delta >0$ such that
\begin{equation}\label{e:Cp-2}
L_a(x,y)=L_b(x,y)=0, \quad \forall b\in(a-\delta,a+\delta)\in {\tilde\tau}_\leq .
\end{equation}
Assume that $x=y=a$. Then $a\in {\tilde X}_\ell$ and 
\begin{equation}\label{e:Cp-3}
L_a(x,y)=L_b(x,y)=0, \quad \forall b\in(-\infty,a]\in {\tilde\tau}_\ell .
\end{equation}
Since  $SC_p({\tilde X}_\ell\times {\tilde X}_\ell)$ carries the topology of pointwise convergence, (\ref{e:Cp-1})--(\ref{e:Cp-3}) imply that the map
\[
\mathcal{L}:(\IR,{\tilde\tau}_\ell)  \to SC_p(\tilde X_\ell\times \tilde X_\ell),\quad \mathcal{L}:a\mapsto L_{a}|_{\tilde X_\ell\times \tilde X_\ell},
\]
is continuous.
Note also that if $a\in {\tilde X}_\ell$ and $b>a$, then $L_b(a,a)=1$ and (see (\ref{e:Cp-3}))
\begin{equation}\label{e:Cp-4}
\mathcal{L}\big((-\infty,a]\big) =\mathcal{L}(\IR) \cap U,
\end{equation}
where  $U=\big\{ f\in SC_p({\tilde X}_\ell\times {\tilde X}_\ell): |f(a,a)|< 1/2 \big\}$ is open in $SC_p({\tilde X}_\ell\times {\tilde X}_\ell)$.
\smallskip

Analogously it can be shown that the function
\[
\mathcal{R}:(\IR,\tilde\tau_r) \to SC_p(\tilde X_r\times \tilde X_r),\quad \mathcal{R}:b\mapsto R_{b}|_{\tilde X_r\times \tilde X_r},
\]
is continuous, and  if $b\in {\tilde X}_r$ and $c<b$, then $R_c(b,b)=1$ and
\begin{equation}\label{e:Cp-5}
\mathcal{R}\big( [b,+\infty)\big) =\mathcal{R}(\IR) \cap V,
\end{equation}
where  $V=\big\{ f\in SC_p({\tilde X}_r\times {\tilde X}_r): |f(b,b)|< 1/2 \big\}$ is open in $SC_p({\tilde X}_r\times {\tilde X}_r)$.

Set $T:= \IR\times SC_p(\tilde X_\ell\times \tilde X_\ell)\times SC_p(\tilde X_r\times \tilde X_r)\in \scp(\M_0^{<\w})$ and define the following map
\[
\mathcal{S}:=(\id,\mathcal{L},\mathcal{R}):(\IR,\tilde\tau) \to T.
\]
Clearly, the map $\mathcal{S}$ is continuous and injective. If $(a,b)\in {\tilde\tau}_\leq $, then
\begin{equation}\label{e:Cp-6}
\mathcal{S}\big( (a,b)\big) =\mathcal{S}(\IR) \cap \big[ (a,b)\times SC_p(\tilde X_\ell\times \tilde X_\ell)\times SC_p(\tilde X_r\times \tilde X_r)\big] .
\end{equation}
If $a\in X_\ell$ and $b\in X_r$, then (\ref{e:Cp-4}) and (\ref{e:Cp-5}) imply
\begin{equation}\label{e:Cp-7}
\mathcal{S}\big( (-\infty,a]\big) =\mathcal{S}(\IR) \cap \big[ \IR \times U\times SC_p(\tilde X_r\times \tilde X_r)\big]
\end{equation}
and
\begin{equation}\label{e:Cp-8}
\mathcal{S}\big( [b,+\infty)\big)=\mathcal{S}(\IR)\cap\big[ \IR\times SC_p(\tilde X_\ell\times \tilde X_\ell)\times V\big] .
\end{equation}
The equalities (\ref{e:Cp-6})--(\ref{e:Cp-8}) imply that the map $\mathcal{S}$  is a topological embedding of the $GO$-space $(\IR,\tilde \tau)$ into the space $T\in \scp(\M_0^{<\w})$, which implies that the $GO$-space $(\IR,\tilde\tau)$ and its subspace $(X,\tau)$  both belong to the class $\scp(\M_0^{<\w})$.
\end{proof}

\section{Proof of Theorem~\ref{t:Cp(M)-iw}}\label{s:up}

Fix any space $X\in \cp^{\M_0\kern-1pt,\Top}[\M_0]$.
By Theorem~\ref{t:Cp-M0-Top}, the space $X$ embeds into the function space $SC_p(Z^{<\w})$ for some separable metrizable space $Z$. Then the space $X\subset SC_p(Z^{<\w})\subset\IR^{Z^{<\w}}$ has weight
$$w(X)\le w(SC_p(Z^{<\w}))\le w(\IR^{Z^{<\w}})\le |Z^{<\w}|\le \mathfrak c.$$

To see that $iw(X)\le\mathfrak c$, choose a countable dense subset $D$ of $Z$ and observe that the restriction operator
\[
R:SC_p(Z^{<\w})\to SC_p(\DN)\subset\IR^{\DN}, \; R:f\mapsto f|_{\DN},
\]
is continuous.

Let us show that $R$ is also injective. Since $R$ is a linear operator, it is enough to show that a function $f\in SC_p(X^n)$ equals zero if $f|_{D^n} =0$. Since $f$ is separately continuous and $D$ is dense in $X$, we have
\[
f|_{X\times D^{n-1}} =f|_{X^2\times D^{n-2}}=\dots = f|_{X^n} =0.
\]
Therefore, the operator $R:X\to \IR^{D^{<\w}}$ is continuous and injective, which implies that
$iw(X)\le w(\IR^{\DN})\le|\DN|\le\aleph_0$ and $|X|\le|\IR^{D^{<\w}}|\le|\IR^\w|=\mathfrak c$.

\section{Proof of Theorem~\ref{t:Cp-Top-Top}}\label{s6}

For a linearly ordered set $X$ by a {\em cut} of $L$ we understand a pair $(A,B)$ consisting of two subsets $A,B$ of $X$ such that $X=A\cup B$ and $a<b$ for any $a\in A$ and $b\in B$. Each point $x\in X$ determines two cuts $x_-=\big((\leftarrow,x),[x,\rightarrow)\big)$ and $x_+=\big((\leftarrow,x],(x,\rightarrow)\big)$ of $X$. The set of all cuts of $X$ will be denoted by $\check X$. The set $\check X$ is linearly ordered by the relation $(A,B)\le (C,D)$ if for each $a\in A$ there is $c\in C$ such that $a\le c$. It is easy to see that the set $X^\pm=\{x_-,x_+:x\in X\}$ is order dense in the linearly ordered set $\check X$.

To each linearly ordered set $L$ we can assign the {\em ordered  group} $H(L)$ consisting of all functions $f:L\to\IQ$ with finite support $\supp(f)=\{x\in L:f(x)\ne 0\}$ and called the {\em Hahn group} (see \cite[1.25]{DW}, where this group is denoted by $\mathfrak F_0(\IQ,L)$). The order on the Hahn group $H(L)$ is defined by letting $f>0$ iff $f\big(\min (\supp(f))\big)>0$. The Hahn group $H^2(L):=H(H(L))$ over the Hahn group $H(L)$ has the structure of an ordered field. For two functions $f,g:H(L)\to\IQ$ in $H^2(\IQ)$ their product in $H^2(L)$ is defined as the convolution $f*g:z\mapsto \sum_{x+y=z}f(x)g(y)$.  By  Theorem~2.15 of \cite{DW}, $H^2(L)$ is a real-closed ordered field of cardinality $\max\{|L|,\aleph_0\}$.

For any (order dense) subspace $D\subset L$ we can identify the Hahn group $H(D)$ with the (order dense) subgroup $\{f\in H(L):\supp(f)\subset D\}$ of $H(L)$ and the field $H^2(D)$ with the (order dense) subfield $\{f\in H^2(L):\supp(f)\subset H(D)\subset H(L)\}$ of the field $H^2(L)$.

The proof of Theorem~\ref{t:Cp-Top-Top} is based on a recent deep result of Chernikov and Shelah \cite[Corollary 2.11]{CS}.

\begin{fact}[Chernikov-Shelah]\label{f:cher-shel} For any infinite cardinal $\kappa$ there are seven linearly ordered spaces $L_0,\dots,L_6$ such that $|L_0|\le\kappa$, $| L_6|\ge 2^\kappa$ and $|L_{k+1}|\le|\check L_k|$ for every $k\in\{0,1,2,3,4,5\}$.
\end{fact}

This result of Chernikov and Shelah will be applied in the proof of the following lemma.

\begin{lemma} \label{l:Cp-TT}
For every ordinal $\alpha<\w_1$, each Tychonoff space of weight $\leq\beth_\alpha$ belongs to the class $\cp[\M_0]$.
\end{lemma}

\begin{proof}
We prove the lemma  by transfinite induction. For $\alpha=0$ the statement is trivially true. Assume that for some countable ordinal $\alpha$ and all ordinals $\beta<\alpha$ we have proved that all Tychonoff spaces of weight $\leq\beth_\beta$ belong to the class
$\cp[\M_0]$. If the ordinal $\alpha$ is limit, then $\beth_\alpha=\sup_{\beta<\alpha}\beth_\beta$. By the inductive assumption, the cardinals $\beth_\beta$, $\beta<\alpha$, endowed with the discrete topology belong to the class $\cp[\M_0]$. Since this class is closed under taking countable topological sums, the cardinal $\beth_\alpha$ belongs to the class $\cp[\M_0]$. Then the Tychonoff power $\IR^{\beth_\alpha}$, being homeomorphic to the function space $C_p(\beth_\alpha)$, belongs to the class $\cp[\M_0]$ too. Since each Tychonoff space of weight $\leq\beth_\alpha$ embeds into $\IR^{\beth_\alpha}\in\cp[\M_0]$, the class $\cp[\M_0]$ contains all Tychonoff spaces of weight $\leq\beth_\alpha$. This completes the inductive step in the case $\alpha$ is a limit ordinal.

Now assume that $\alpha=\beta+1$ is a successor ordinal. Then $\beth_\alpha=2^{\beth_\beta}$. By Fact~\ref{f:cher-shel}, there are infinite linearly ordered sets $L_0,\dots,L_6$ such that $|L_0|\leq\beth_\beta$, $\beth_\alpha \leq |L_6|$, and $|L_{k+1}|\le |\check L_{k}|$ for all $k\in\{0,\dots,5\}$. By induction, for every $k\in\{0,\dots,6\}$ we shall prove that each Tychonoff space of weight $\leq |L_k|$ belongs to the class $\cp[\M_0]$.
For $k=0$ this follows from the inequality $|L_0|\leq\beth_\beta$ and the inductive assumption. Assume that for some number $k\in\{0,\dots,6\}$ we have proved that every Tychonoff space of weight $\leq|L_{k}|$ belongs to the class $\cp[\M_0]$. Consider the linearly ordered space $\check L_{k}$, and observe that $L_{k}^\pm$ is an order dense subset of $\check L_{k}$ and
\[
\{(\leftarrow,x):x\in L^\pm_{k}\}\cup\{(x,\rightarrow):x\in L_{k}^\pm\}
\]
is a base of the order topology of $\check L_{k}$. So, the linearly ordered space $\check L_{k}$ has topological weight $\leq|L_{k}|$. Now consider the Hahn field $H^2(\check L_{k})$ over the linearly ordered set $\check L_k$ and its order dense subfield $H^2(L_{k}^\pm)$.
Taking into account that the weight of an ordered field is equal to its density, we obtain that the ordered field $F_{k}:=H^2(\check L_{k})$ has topological weight
\[
w(F_{k})\leq|H^2(L_{k}^\pm)|=|L_{k}^\pm|=|L_{k}|.
\]
By the inductive assumption, the space $F_{k}$ belongs to the  class $\cp[\M_0]$ and so does the function space $C_p(F_k,C_p(F_k,F_k))$. This function space is canonically homeomorphic to the subspace \mbox{$SC_p(F_k\times F_k,F_k)$} of $F_k^{F_k\times F_k}$ consisting of separately continuous functions.

By Lemma \ref{lemma:SCp}, the space $SC_p(F_k\times F_k,F_k)\in\cp[\M_0]$ contains a discrete subspace $D$ of cardinality $|F_k|$. Observe that
\[
|D|=|F_k|=|H^2(\check L_k)|=|\check L_k|\ge |L_{k+1}|.
\]
Hence $\cp[\M_0]$ contains also the function space $C_p(D)$, which is homeomorphic to $\IR^D$.
Since each Tychonoff space of cardinality $\le|L_{k+1}|\le|D|$ embeds into $\IR^D$, the class $\cp[\M_0]$ contains all Tychonoff spaces of cardinality $\le|L_{k+1}|$. This completes the inductive step. For $k=6$ we get $|L_6|\geq \beth_\alpha$, which implies that the class $\cp[\M_0]$ contains all Tychonoff spaces of weight $\leq\beth_\alpha$.
 \end{proof}

Now we are ready to prove Theorem~\ref{t:Cp-Top-Top}.
\begin{proof}[Proof of Theorem \ref{t:Cp-Top-Top}]
Denote by $\Tych(\beth_{\w_1})$ the class of all Tychonoff spaces of weight strictly less than $\beth_{\w_1}$. Clearly, $\Tych(\beth_{\w_1})$ is stable. Let $X,Y\in \Tych(\beth_{\w_1})$. Then $\kappa=\max\{w(X),w(Y)\}<\beth_{\w_1}$, $|X|\leq 2^{w(X)}\leq 2^\kappa$ and
\[
w(C_p(X,Y))\leq w(Y^X)\leq w(Y)^{|X|}\leq \kappa^{|X|}\leq 2^{\kappa\cdot |X|}\leq 2^{2^\kappa}<\beth_{\w_1},
\]
which implies $C_p(X,Y)\in\Tych(\beth_{\w_1})$. Therefore, the class $\Tych(\beth_{\w_1})$ is $C_p$-stable and hence
$\cp[\M_0]\subset \Tych(\beth_{\w_1})$. On the other hand, Lemma~\ref{l:Cp-TT} guarantees that
$\Tych(\beth_{\w_1})\subset\cp[\M_0]$ and hence $\cp[\M_0]=\Tych(\beth_{\w_1})$.

Since each Tychonoff space $X$ has weight $w(X)\leq 2^{|X|}$ and cardinality $|X|\leq 2^{w(X)}$, the inequalities $w(X)<\beth_{\w_1}$ and $|X|<\beth_{\w_1}$ are equivalent, which implies that  the class $\Tych(\beth_{\w_1})$ coincides with the class of all Tychonoff spaces of cardinality $< \beth_{\w_1}$.
\end{proof}

\section{Open Problems}

By Theorem~\ref{t:lower}, the class $\scp(\M_0^{<\w})$ contains all cosmic spaces and all metrizable spaces of weight $\le\mathfrak c$. The classes of cosmic spaces and of metrizable spaces are contained in the class of $\sigma$-spaces. 
It can be shown that each $\sigma$-space $X$ of weight $\le\mathfrak c$ admits a continuous injective map onto a metrizable space of weight $\le\mathfrak c$ and hence has countable $i$-weight.

\begin{problem}
Does the class $\scp(\M_0^{<\w})$ contain all $\sigma$-spaces of weight $\le\mathfrak c$?
\end{problem}

It is easy to see that for every metrizable separable space $X$ the space $SC_p(X^{<\w})$ is canonically homeomorphic to the countable product $\prod_{n\in\IN}SC_p(X^n)$. We can ask how different are the function spaces $SC_p(X^n)$ for various $n\in\IN$.

\begin{problem}
Is it true that for every positive integers $n<m$ the space $SC_p([0,1]^m)$ does not embed into the space $SC_p([0,1]^n)$? {\rm (The answer is affirmative if $n=1$ as $SC_p([0,1])=C_p([0,1])$ is cosmic while $SC_p([0,1]^m)$ is not).}
\end{problem}

The affirmative answer to the following problem would simplify the
proof of Theorem~\ref{t:Cp-Top-Top}. Observe that this problem has affirmative answer under Generalized Continuum Hypothesis.

\begin{problem} Is it true that for every infinite cardinal $\kappa$ there exists a Tychonoff space $X$ of weight $w(X)\le\kappa$ such that the function space $SC_p(X\times X)$ contains a discrete subspace of cardinality $2^\kappa$?
\end{problem}

Using Fact \ref{f:Cosmic} and (the proof of) Theorem~\ref{t:Cp-M0-M0} one can show for the class $\mathfrak{C}$ of all cosmic spaces we get $\cp^{\mathfrak{C},\Top}[\M_0]=\cp^{\M_0\kern-1pt,\Top}[\M_0]= \scp\left(\MN_0\right)$. It would be interesting to determine the extensions $\cp^{\X,\Top}[\M_0]$ of the class $\M_0$ for some other classes $\X$ of topological spaces, in particular, for the class $\mathfrak{L}$ of Lindel\"{o}f spaces.
\begin{problem}
Describe the class $\cp^{\mathfrak L,\Top}[\M_0]$. Does it contain non-metrizable compacta?
\end{problem}

\end{document}